\documentclass[11pt]{amsart}

\usepackage[margin=1.25in]{geometry}

\usepackage{graphicx}

\usepackage{wasysym}
\usepackage[misc]{ifsym}

\usepackage{amsmath, amssymb, amsthm, graphicx, enumerate, tikz, float, bbm, nicefrac, mathrsfs}
\usepackage{mathtools, comment}
\usepackage{subcaption}
\usepackage[colorlinks]{hyperref}
\hypersetup{citecolor=blue}
\usetikzlibrary{matrix,arrows,decorations.pathmorphing}

\newtheorem{theorem}{Theorem}[section]
\newtheorem{lemma}[theorem]{Lemma}
\newtheorem{corollary}[theorem]{Corollary}

\newcommand{\n}[1]{||#1||}

\theoremstyle{definition}
\newtheorem{definition}[theorem]{Definition}
\newtheorem{example}[theorem]{Example}

\newtheorem*{solution*}{Solution}

\theoremstyle{remark}
\newtheorem{remark}[theorem]{Remark}
\newtheorem{claim}[theorem]{Claim}
\newtheorem*{pf*}{Pf}
\newtheorem*{pfbase*}{Proof of Base Case}
\newtheorem*{pfstep*}{Proof of Inductive Step}

\numberwithin{equation}{section}

\newcommand{\C}{\mathbb{C}}

\newcommand{\Z}{\mathbb{Z}}
\newcommand{\U}{\mathscr{U}}
\newcommand{\M}{\mathscr{M}}

\begin{document}

\allowdisplaybreaks

\title[$G$-Invariant Spaces of Functions]{Spaces of Continuous and Measurable Functions Invariant Under a Group Action}

\author{Samuel A. Hokamp}
\address{Department of Mathematics, Sterling College, Sterling, Kansas, 67579}
\email{samuel.hokamp@sterling.edu}


\date{\today}

\keywords{Spaces of continuous functions, group actions, functional analysis.\\\indent\emph{Corresponding author.} Samuel A. Hokamp \Letter~\href{mailto:samuel.hokamp@sterling.edu}{samuel.hokamp@sterling.edu}. \phone~920-634-7356.}

\begin{abstract}
In this paper we characterize spaces of continuous and $L^p$-functions on a compact Hausdorff space that are invariant under a transitive and continuous group action. This work generalizes Nagel and Rudin's 1976 results concerning unitarily and M\"obius invariant spaces of continuous and measurable functions defined on the unit sphere in $\C^n$.
\end{abstract}

\maketitle

\section{Introduction}

The idea for this paper came from the realization that much of Nagel and Rudin's work characterizing unitarily invariant spaces of continuous and measurable functions on the unit sphere of $\mathbb{C}^n$ (originally found in \cite{NR} and summarized in \cite{RFT}) could be generalized to spaces of continuous and measurable functions on a compact Hausdorff space $X$, which are invariant under the continuous and transitive action of a compact group $G$ on $X$. 


A space of complex functions on $X$ is \textit{$G$-invariant} if the pre-composition of any function in the set with the action of each element of $G$ on $X$ remains in the set. A $G$-invariant space is \textit{$G$-minimal} if it contains no proper $G$-invariant subspace. Our main result (Theorem \ref{main result}) yields that a collection of closed $G$-minimal spaces of continuous functions satisfying certain conditions suffices to characterize all closed $G$-invariant spaces of continuous functions on $X$: each closed $G$-invariant space is the closure of the direct sum of a unique subcollection of the $G$-minimal spaces. 

A unique regular Borel probability measure $\mu$ on $X$ that is $G$-invariant in the sense that $$\int_X f\,d\mu=\int_X f(\alpha x)\,d\mu(x),$$ for every continuous function $f$ on $X$ and every $\alpha\in G$ is necessary to define the conditions for the collection of closed $G$-minimal spaces. Existence of such a measure is due to Andr\'e Weil in \cite{weil1940}. Additionally, Theorem \ref{main result} shows each closed $G$-invariant space of \textit{measurable} functions with respect to $\mu$ is characterized by the same collection of $G$-minimal spaces.

In Section \ref{G Collections}, we define the conditions which a collection of $G$-minimal spaces must have in order to induce the closed $G$-invariant spaces of continuous and measurable functions. In Section \ref{Main Result Section}, we prove our main result, Theorem \ref{main result}. Section \ref{lemmas} is devoted to the proofs of Lemma \ref{Lemma2} and Lemma \ref{Lemma3}, which are used in establishing Theorem \ref{main result}.

\section{Preliminaries} 


Let $X$ be a compact Hausdorff space and $C(X)$ the space of continuous complex functions with domain $X$. Let $G$ be a compact group that acts continuously and transitively on $X$. When we wish to be explicit, the map $\varphi_\alpha:X\to X$ shall denote the action of $\alpha$ on $X$ for each $\alpha\in G$; otherwise, $\alpha x$ denotes the action of $\alpha\in G$ on $x\in X$.

Let $\mu$ denote the unique regular Borel probability measure on $X$ that is invariant under the action of $G$. Specifically,
\begin{equation}\label{measinv}
    \int_X f\,d\mu=\int_X f\circ\varphi_\alpha\,d\mu,
\end{equation}
for all $f\in C(X)$ and $\alpha\in G$. The existence of such a measure is a result of Andr\'e Weil from \cite{weil1940}. A construction of $\mu$ can be found in \cite{joys} (Theorem~6.2), but existence can be established using the Riesz Representation Theorem (for reference, Theorem 6.19 \cite{RRC}). Throughout the paper, $\mu$ shall refer to this measure.

The notation $L^p(\mu)$ denotes the usual Lebesgue spaces, for $1\leq p\leq\infty$. For $Y\subset C(X)$, the uniform closure of $Y$ is denoted $\overline{Y}$, and for $Y\subset L^p(\mu)$, the norm-closure of $Y$ in $L^p(\mu)$ is denoted $\overline{Y}^{p}$.

The following is an easy consequence of \eqref{measinv}:

\begin{remark}\label{switcher-mu}
Let $1\leq p<\infty$ and let $p^\prime$ be its conjugate exponent. Then $$\int_X (f\circ\varphi_\alpha)\cdot g\,d\mu=\int_X f\cdot (g\circ\varphi_{\alpha^{-1}})\,d\mu,$$ for $f\in L^p(\mu)$, $g\in L^{p^\prime}(\mu)$, and $\alpha\in G$.
\end{remark}

The following definitions are generalizations of definitions found in \cite{RFT} related to the unitary group. These more specific definitions are given as references.

\begin{definition}[12.2.4 \cite{RFT}]
A space of complex functions $Y$ defined on $X$ is \textbf{invariant under $G$} (\textbf{$G$-invariant}) if $f\circ \varphi_\alpha\in Y$ for every $f\in Y$ and every $\alpha\in G$.
\end{definition}

\begin{remark}
Since the action is continuous, $C(X)$ is $G$-invariant. Conversely, if $C(X)$ is $G$-invariant, then each action $\varphi_\alpha$ must be continuous.
\end{remark}

\begin{remark}
Explicitly, the invariance property \eqref{measinv} means $\mu(\alpha E)=\mu(E)$ for every Borel set $E$ and every $\alpha\in G$. Consequently, \eqref{measinv} holds for every $L^p$ function, and thus $L^p(\mu)$ is $G$-invariant for all $1\leq p\leq\infty$.
\end{remark}

\begin{definition}[12.2.4 \cite{RFT}]
If $Y$ is $G$-invariant and $T$ is a linear transformation on $Y$, we say $T$ \textbf{commutes with} $G$ if $$T(f\circ\varphi_\alpha)=(Tf)\circ\varphi_\alpha$$ for every $f\in Y$ and every $\alpha\in G$.
\end{definition}

\begin{definition}[12.2.8 \cite{RFT}]
A space $Y\subset C(X)$ is \textbf{$G$-minimal} if it is $G$-invariant and contains no nontrivial $G$-invariant spaces.
\end{definition}

\begin{example}\label{ex part1}
To illustrate these definitions, let $X=G=T^n$, the torus in $\C^n$, such that the action of $G$ on $X$ is given by coordinatewise multiplication. This action is both transitive and continuous.

For each $k=(k_1,k_2,\ldots,k_n)\in\Z^n$, we define $H_k$ to be the space of all complex functions $f$ on $T^n$ given by $f(z)=cz^k$, where $c\in\C$ and $z^k=z_1^{k_1} z_2^{k_2}\ldots z_n^{k_n}$; that is, $H_k$ is the span of the trigonometric monomial of power $k$.

Observe that $\dim H_k=1$, so that each $H_k$ is closed. Further, $G$-invariance of each $H_k$ is clear, and thus each $H_k$ is $G$-minimal.
\end{example}

Finally, the classical results used in this paper can be found in many texts, with the reference given being one such place.

\section{\texorpdfstring{$G$}{G}-Collections}\label{G Collections}

In this section, we introduce the notion of a \textit{$G$-collection}, a collection of closed $G$-invariant spaces which characterize all closed $G$-invariant subspaces of $C(X)$ (Definition \ref{G-col def1}). However, we must first give Definition \ref{H(x)spaces}. The particular case of the unitary group acting on the unit sphere in $\mathbb{C}^n$ described in \cite{RFT} inherently satisfies Definition \ref{G-col def1}($*$), so has no need of the following definition, but this is not necessarily true in general for $G$ acting on $X$. 

\begin{definition}\label{H(x)spaces}
For each $x\in X$, the space $H(x)$ is the set of all continuous functions that are unchanged by the action of any element of $G$ which stabilizes $x$. That is, $$H(x)=\{f\in C(X):f=f\circ\varphi_\alpha,\text{ for all }\alpha\in G\text{ such that }\alpha x=x\}.$$
\end{definition}

\begin{definition}\label{G-col def1}
Let $\mathscr{G}$ be a collection of spaces in $C(X)$ with the following properties:
\begin{itemize}
    \item[(1)] Each $H\in\mathscr{G}$ is a closed $G$-minimal space.
    \item[(2)] Each pair $H_1$ and $H_2$ in $\mathscr{G}$ is orthogonal (in $L^2(\mu)$): If $f_1\in H_1$ and $f_2\in H_2$, then $$\int_X f_1\bar{f_2}\,d\mu=0.$$
    \item[(3)] $L^2(\mu)$ is the direct sum of the spaces in $\mathscr{G}$.
\end{itemize}
We say $\mathscr{G}$ is a \textbf{$G$-collection} if it also possesses the following property:
\begin{itemize}
    \item[($*$)] $\dim (H\cap H(x))=1$ for each $x\in X$ and each $H\in\mathscr{G}$.
\end{itemize}
\end{definition}



\begin{remark}
A collection of spaces in $C(X)$ lacking at most only property $(*)$ of Definition~\ref{G-col def1} always exists, as a consequence of the Peter-Weyl theorem from \cite{peterweyl}. This collection is necessarily unique.
\end{remark}

\begin{remark}
Explicitly, Definition \ref{G-col def1}(3) requires each $f\in L^2(\mu)$ to have a unique expansion $f=\sum f_i$, with $f_i\in H_i$, that converges unconditionally to $f$ in the $L^2$-norm.
\end{remark}

Throughout the remainder of the paper, we assume that a $G$-collection $\mathscr{G}$ exists for $X$, indexed by $I$. The rest of this section is devoted to establishing results related to $\mathscr{G}$ and its elements $H_i$, beginning with the following theorem, which is a generalization of Theorem 12.2.5 of \cite{RFT}. Note that we use $[\cdot,\cdot]$ to denote the inner product on $L^2(\mu)$: $$[f,g]=\int_X f\bar{g}\,d\mu.$$

\begin{theorem}\label{proj kernels}
Suppose $H$ is a closed $G$-invariant subspace of $C(X)$, and $\pi$ is the orthogonal projection of $L^2(\mu)$ onto $H$. Then, $\pi$ commutes with $G$, and to each $x\in X$ corresponds a unique $K_x\in H$ such that 
\begin{equation}\label{eq2}
(\pi f)(x)=[f,K_x]\hspace{.25in}(f\in L^2(\mu)).
\end{equation}
Additionally, the functions $K_x$ satisfy the following:
\begin{itemize}
    \item[(1)] $\displaystyle K_x(y)=\overline{K_y(x)}$\hspace{.5in}$(x,y\in X)$,
    \item[(2)] $\displaystyle\pi f=\int_X f(x)K_x\,d\mu(x)$\hspace{.5in}$(f\in L^2(\mu))$,
    \item[(3)] $K_{\varphi_\alpha( x)}=K_x\circ\varphi_{\alpha^{-1}}$\hspace{.5in}$(\alpha\in G)$,
    \item[(4)] $K_x=K_x\circ\varphi_\alpha$, for all $\alpha\in G$ such that $\alpha x=x$, and
    \item[(5)] $K_x(x)=K_y(y)>0$\hspace{.5in}$(x,y\in X)$.
\end{itemize}
\end{theorem}

\begin{proof}
The projection $\pi$ commutes with $G$ due to the $G$-invariance of $H^\perp$, which follows from Corollary \ref{switcher-mu}. The existence and uniqueness of $K_x$ follows from the fact that $f\mapsto(\pi f)(x)$ is a bounded linear functional on $L^2(\mu)$. Further, $K_x\in H$ since $\pi f=0$ whenever $f\perp H$. When $f\in H$, we get $$f(x)=[f,K_x].$$ In particular, $K_y(x)=[K_y,K_x]$, which proves (1), and (2) follows naturally. Since $\pi$ commutes with $G$, $$[f,K_{\varphi_\alpha(x)}]=(\pi f)(\varphi_\alpha(x))=\pi(f\circ\varphi_\alpha)(x)=[f\circ\varphi_\alpha,K_x]=[f,K_x\circ\varphi_{\alpha^{-1}}],$$ for every $f\in L^2(\mu)$ (Corollary \ref{switcher-mu} yields the last equality). This proves (3) and the special case (4). Finally, (3) also yields $$K_{\varphi_\alpha(x)}(\varphi_\alpha(x))=(K_x\circ\varphi_{\alpha^{-1}})(\varphi_\alpha(x))=K_x(x).$$ This and the transitivity of the group action yields (5), with the inequality due to $$K_x(x)=[K_x,K_x]>0.$$
\end{proof}

\begin{remark}
Theorem~\ref{proj kernels}(4) yields that $\dim (H(x)\cap H_i)\geq 1$ for each $x\in X$ and $i\in I$, so that Definition \ref{G-col def1}($*$) requires each $H_i$ to contain a unique (up to a constant multiple) function which satisfies Theorem~\ref{proj kernels}(4) for each $x\in X$.
\end{remark}

\begin{definition}
We define $\pi_i$ to be the projection of $L^2(\mu)$ onto $H_i$. The domain of each $\pi_i$ is extended to $L^1(\mu)$ by Theorem \ref{proj kernels}(2). 
\end{definition}

\begin{definition}
If $\Omega\subset I$, we define $E_\Omega$ to be the direct sum of the spaces $H_i$ for $i\in\Omega$. 
\end{definition}

\begin{theorem}
Suppose $T:H_i\to H_j$ is linear and commutes with $G$. When $i=j$, $T$ is the identity on $H_i$ scaled by a constant $c$. Otherwise, $T=0$.
\end{theorem}

\begin{proof}
For each $x\in X$, let $K_x$ denote the kernel of Theorem \ref{proj kernels} in $H_i$ and $L_x$ the same in $H_j$. Then, if $\alpha\in G$ such that $\alpha x=x$, because $T$ commutes with $G$, we get $$TK_x=T(K_x\circ\varphi_\alpha)=(TK_x)\circ\varphi_\alpha.$$ Thus, by Definition \ref{G-col def1}($*$), $TK_x=c(x)L_x$ for some constant $c(x)$, and hence $$(TK_x)(x)=c(x)L_x(x).$$ Observe that $L_x(x)$ is independent of $x$. Further, if $y=\alpha x$, then $$(TK_y)(y)=(TK_x\circ\varphi_{\alpha^{-1}})(\alpha x)=(TK_x)(x).$$ Thus, $c(x)=c$ is the same constant for all $x\in X$.

If $f\in H_i$, we then get $$f=\int_X f(x)K_x\,d\mu(x).$$ Application of $T$ yields $$Tf=\int_X f(x)TK_x\,d\mu(x)=c\int_X f(x)L_x\,d\mu(x)=c\pi_j f.$$ When $i=j$, then $\pi_j f=f$ for all $f\in H_i$. When $i\neq j$, then $\pi_j f=0$ for all $f\in H_i$.
\end{proof}

\section{Characterization of Closed \texorpdfstring{$G$}{G}-Invariant Spaces}\label{Main Result Section}

We now prove our main result, Theorem \ref{main result}. Throughout the section, we let $\mathscr{X}$ denote any of the spaces $C(X)$ or $L^p(\mu)$ for $1\leq p<\infty$.

\begin{theorem}\label{main result}
If $Y$ is a closed $G$-invariant subspace of $\mathscr{X}$, then $Y$ is the closure of the direct sum of some subcollection of $\mathscr{G}$.
\end{theorem}

The proof of Theorem \ref{main result} relies on the particular case when $\mathscr{X}$ is the space $L^2(\mu)$ (Theorem \ref{L2 result}), as well as Lemma \ref{Lemma2} and Lemma \ref{Lemma3}, which allow the passage from $L^2(\mu)$ to the other spaces. These lemmas are proved in Section \ref{lemmas}.

\begin{theorem}\label{L2 result}
If $Y$ is a closed $G$-invariant subspace of $L^2(\mu)$, 
then $Y$ is the $L^2$-closure of the direct sum of some subcollection of $\mathscr{G}$.
\end{theorem}

\begin{proof}
Define the set $\Omega=\{i\in I:\pi_i Y\neq\{0\}\}$ and let $i\in\Omega$. Since $Y$ is $G$-invariant and $\pi_i$ commutes with $G$, $\pi_i Y$ is a nontrivial $G$-invariant space in $H_i$. The $G$-minimality of $H_i$ then yields that $\pi_i Y=H_i$.

Let $Y_0$ be the null space of $\pi_i$ in $Y$, with relative orthogonal complement $Y_1$. Then $Y_0$ is $G$-invariant, and so is $Y_1$. Further, $\pi_i:Y_1\to H_i$ is an isomorphism, whose inverse we denote $\Lambda$. If we fix $j\in I$ such that $j\neq i$ and define $T=\pi_j\circ\Lambda$, then $T$ maps $H_i$ into $H_j$ and commutes with $G$. Thus, $T=0$.

We conclude that $\pi_jY_1=\{0\}$ for all $j\neq i$, and thus $Y_1=H_i$. Thus, $H_i\subset Y$ for all $i\in\Omega$, and further, $\overline{E}_\Omega^2\subset Y$. Since $\pi_j Y=\{0\}$ for all $j\notin\Omega$, Definition~\ref{G-col def1}(3) yields the opposite inclusion.
\end{proof}

\begin{lemma}\label{Lemma2}
If $Y$ is a closed $G$-invariant space in $\mathscr{X}$, then $Y\cap C(X)$ is dense in $Y$.
\end{lemma}

\begin{lemma}\label{Lemma3}
If $Y\subset C(X)$, $Y$ is a $G$-invariant space, and some $g\in C(X)$ is not in the uniform closure of $Y$, then $g$ is not in the $L^2$-closure of $Y$. 
\end{lemma}

\begin{proof}[Proof of Theorem \ref{main result}]
If $Y$ is a closed $G$-invariant subspace of $\mathscr{X}$, define $\tilde{Y}$ to be the $L^2$-closure of $Y\cap C(X)$. Lemma \ref{Lemma3} then yields $$\tilde{Y}\cap C(X)=Y\cap C(X).$$ We next observe that $Y\cap C(X)$ is $L^2$-dense in $\tilde{Y}$ and is $\mathscr{X}$-dense in $Y$, by Lemma \ref{Lemma2}. Each $\pi_i$ is $\mathscr{X}$-continuous as well as $L^2$-continuous, so that $\pi_i Y=\{0\}$ if and only if $\pi_i\tilde{Y}=\{0\}$. By Theorem \ref{L2 result}, $\tilde{Y}$ is the $L^2$-closure of $E_\Omega$, where $\Omega$ is the set of all $i\in I$ such that $\pi_i Y\neq\{0\}$. Another application of Lemma \ref{Lemma3} yields $$\tilde{Y}\cap C(X)=\overline{E}_\Omega.$$ The $\mathscr{X}$-density of $Y\cap C(X)$ in $Y$ then implies $Y$ is the $\mathscr{X}$-closure of $E_\Omega$.
\end{proof}

\begin{example}
We now further explore the situation that was set up in Example \ref{ex part1}. Recall that $X=G=T^n$, the torus in $\C^n$, such that the action of $G$ on $X$ is given by coordinatewise multiplication. This action is both transitive and continuous, and the measure induced by the action is the usual Lebesgue measure $m$, normalized so that $m(T^n)=1$.

Further, $H_k$ is the space of all complex functions $f$ on $T^n$ given by $f(z)=cz^k$, where $c\in\C$ and $z^k=z_1^{k_1} z_2^{k_2}\ldots z_n^{k_n}$ for  $k=(k_1,k_2,\ldots,k_n)\in\Z^n$; that is, $H_k$ is the span of the trigonometric monomial of power $k$.

The collection $\mathscr{G}$ of spaces $H_k$ forms a $G$-collection: Each $H_k$ is a closed $G$-invariant space of dimension 1, thus is $G$-minimal. Further, $$\int_{T^n}z^k\bar{z}^{k'}\,dm(z)=\begin{cases} 1 & \text{if }k=k' \\ 0 & \text{if }k\neq k', \end{cases}$$ so that the spaces $H_k$ are pairwise orthogonal. Finally, $L^2(T^n)$ is the direct sum of the spaces $H_k$ as a consequence of the Stone-Weierstrass theorem (presented in \cite{RFA} as a special case of Bishop's Theorem, Theorem 5.7). Thus, $\mathscr{G}$ satisfies the first three properties of Definition \ref{G-col def1}. Lastly, for $z\in T^n$, we have $\dim (H(z)\cap H_k)=1$ from Theorem \ref{proj kernels} and the fact that $\dim H_k=1$. Thus, $\mathscr{G}$ is a $G$-collection.

Theorem \ref{main result} then yields that every closed $G$-invariant space of continuous or $L^p$ functions on $T^n$ is the closure of the direct sum of some collection of spaces $H_k$. Notably, the collection which induces the space of all functions which are restrictions to $T^n$ of functions holomorphic on the polydisc and continuous on the closed polydisc is the collection of all $H_k$ such that $k$ has nonnegative coordinates.
\end{example}

\section{Proofs of Lemma \ref{Lemma2} and Lemma \ref{Lemma3}}\label{lemmas}


As in Section \ref{Main Result Section}, we let $\mathscr{X}$ denote any of the spaces $C(X)$ or $L^p(\mu)$, for $1\leq p<\infty$. The proofs of Lemma \ref{Lemma2} and Lemma \ref{Lemma3} (given at the end of the section) require Lemma \ref{Lemma0} and Lemma \ref{Lemma1}, which we now prove.

\begin{lemma}\label{Lemma0}
If $f\in C(X)$, then the map $\alpha\mapsto f\circ\varphi_\alpha$ is a continuous map of $G$ into $C(X)$.
\end{lemma}

\begin{proof}
For $\alpha\in G$, we define the map $\phi:G\to C(X,X)$ by $\phi(\alpha)=\varphi_\alpha$. 
Then $\phi$ is continuous when $C(X,X)$ is given the compact-open topology (Theorem 46.11 of \cite{MUNK}). We note that the continuity of the group action is used here.

We define the map $T_f:C(X,X)\to C(X)$ for $f\in C(X)$ by $T_f(\varphi)=f\circ\varphi$, and we endow both spaces with the compact-open topology. Let $f\circ\varphi\in C(X)$ for $\varphi\in C(X,X)$ and suppose $f\circ\varphi\in V$, where $V=V(K,U)$ is a subbasis element in $C(X)$. Explicitly, $$K\subset (f\circ\varphi)^{-1}(U). \hspace{.2in}\text{ That is to say, }\hspace{.2in} K\subset \varphi^{-1}(f^{-1}(U)).$$

Then $V'=V'(K,f^{-1}(U))$ is a subbasis element in $C(X,X)$ and $\varphi\in V'$. Further, $V'\subset T_f^{-1}(V)$, so that $T_f$ is continuous when $C(X,X)$ and $C(X)$ are endowed with the respective compact-open topologies. We finally observe that since $X$ is compact, the norm topology and the compact-open topology on $C(X)$ coincide.
\end{proof}

\begin{lemma}\label{Lemma1}
If $f\in\mathscr{X}$, then the map $\alpha\mapsto f\circ\varphi_\alpha$ is a continuous map of $G$ into $\mathscr{X}$.
\end{lemma}

\begin{proof}
For brevity, we let $\n{\cdot}$ denote the norm of the space $\mathscr{X}$ and $\n{\cdot}_\infty$ the uniform norm in $C(X)$. If $\epsilon>0$, then $\n{f-g}<\epsilon/3$ for some $g\in C(X)$. There is a neighborhood $N$ of the identity in $G$ such that $\n{g-g\circ\varphi_\alpha}_\infty<\epsilon/3$ for all $\alpha\in N$ (Lemma \ref{Lemma0}). Since $$|f-f\circ\varphi_\alpha|\leq|f-g|+|g-g\circ\varphi_\alpha|+|(g-f)\circ\varphi_\alpha|,$$ we have $\n{f-f\circ\varphi_\alpha}<\epsilon$ for all $\alpha\in N$.
\end{proof}












\begin{proof}[Proof of Lemma \ref{Lemma2}]
Let $f\in Y$ and choose $N$ as in the proof of Lemma \ref{Lemma1}. Let $\psi:G\to[0,\infty)$ be continuous, with support in $N$, such that $\int\psi\,dm=1$, where $m$ denotes the Haar measure on $G$. Define $$g(x)=\int_G\psi(\alpha)f(\alpha x)\,dm(\alpha).$$ Since $\alpha\mapsto\psi(\alpha)f\circ\varphi_\alpha$ is a continuous map into $Y$, we have $g\in Y$. If $\beta\in G$ such that $\beta x_1=x$, the invariance of the Haar measure yields $$g(x)=\int_G\psi(\alpha\beta^{-1})f(\alpha x_1)\,dm(\alpha).$$ Thus, $g\in Y\cap C(X)$.

Finally, the relation $$f-g=\int_N\psi(\alpha)(f-f\circ\varphi_\alpha)\,dm(\alpha)$$ gives $\n{f-g}<\epsilon$, since $\n{f-f\circ\varphi_\alpha}<\epsilon$ whenever $\alpha\in N$.
\end{proof}

\begin{proof}[Proof of Lemma \ref{Lemma3}]
There is a $\mu'\in M(X)$ such that $\int f\,d\mu'=0$ for all $f\in Y$, but $\int g\,d\mu'=1$. There is a neighborhood $N$ of the identity in $G$ such that $\text{Re}\int g\circ\varphi_\alpha\,d\mu'>\frac{1}{2}$ for every $\alpha\in N$. Associate $\psi$ to $N$ as in the proof of Lemma \ref{Lemma2}, and define $\Lambda\in C(X)^*$ by $$\Lambda h=\int_X\int_G\psi(\alpha)h(\alpha x)\,dm(\alpha)\,d\mu'(x).$$ By the Schwarz inequality, $$\Big|\int_G\psi(\alpha)h(\alpha x)\,dm(\alpha)\Big|^2\leq \int_G|\psi(\alpha)|^2\,dm(\alpha)\int_G|h(\alpha x)|^2\,dm(\alpha)=\n{\psi}_2^2\int_X|h|^2\,d\mu,$$ so that $$|\Lambda h|\leq\n{\mu'}\text{ }\n{\psi}_2\n{h}_2.$$ Thus, $\Lambda$ extends to a bounded linear functional $\Lambda_1$ on $L^2(\mu)$. By interchanging integrals in the definition of $\Lambda$, we get $\Lambda_1 f=0$ for every $f\in Y$, but $\text{Re }\Lambda_1 g\geq\frac{1}{2}$. Thus, $g\notin\overline{Y}^2$.
\end{proof}

\section{Future Questions}

\begin{itemize}
    \item[(1)] Does a $G$-collection exist for all groups $G$ acting continuously and transitively on $X$? What about a collection that only lacks ($*$)? What conditions might exist on $G$ or $X$ that yield a collection that only lacks ($*$)?
    \item[(2)] Under what conditions can the restrictions on $X$, $G$, and the action of $G$ on $X$ be loosened? Can the compactness of $X$ and $G$ be substituted with local compactness? Can the continuity of the action be substituted with separate continuity?
    \item[(3)] Suppose $H$ is a subgroup of $G$ and $\mathscr{H}$ is a collection of closed $H$-minimal spaces satisfying the same conditions as $\mathscr{G}$. What is the relationship between $\mathscr{H}$ and $\mathscr{G}$? The uniqueness of $\mu$ shows that the $H$-measure is the same as the $G$-measure, and further, $G$-invariance implies $H$-invariance (of a space).
    
    We note that (3) is prompted from the study of $\M$-invariant and $\U$-invariant spaces of continuous functions on the unit sphere of $\mathbb{C}^n$ from \cite{NR}, in which it is shown that there are infinitely many $\U$-invariant spaces and only six $\M$-invariant spaces. These six $\M$-invariant spaces are found by combining the $\U$-minimal spaces in a specific way (see Lemma 13.1.2 of \cite{RFT}), and we are curious if this method can be generalized.
    \item[(4)] Can the results of \cite{hokamp2021certain} similarly be generalized? That is, can a $G$-collection similarly characterize all weak*-closed $G$-invariant subspaces of $L^\infty(\mu)$?
    \item[(5)] Under what conditions can a $G$-collection characterize the closed $G$-invariant \textit{algebras} of continuous functions? We note that the case for the unitary group acting on the unit sphere of $\mathbb{C}^n$ is discussed in \cite{RUA} and is also summarized in \cite{RFT}.
\end{itemize}

\section{Data Availability Statement}

Data sharing not applicable to this article as no datasets were generated or analysed during the current study.

\bibliographystyle{unsrt}
\bibliography{bibliography}

\end{document}